\documentclass[11pt]{amsart}

\usepackage[foot]{amsaddr}

\usepackage{amssymb,amsfonts,amsmath,graphicx,verbatim,eufrak,amsthm,calc}
\usepackage{color}
\usepackage{xcolor}
\usepackage{newfloat}
\usepackage{hyperref}
\usepackage{faktor}

\usepackage{tikz}
\hypersetup{
hidelinks,backref=true,pagebackref=true,hyperindex=true,
    colorlinks=true,       
    linkcolor=violet,          
    citecolor=green,        
    filecolor=magenta,      
    urlcolor=cyan           
}

\colorlet{genial}{black} 
\colorlet{genialsol}{black}

\makeatletter

\newtheoremstyle{genialnumbox}
{7pt}
{7pt}
{\normalfont}
{}
{\small\bf\sffamily\color{genial}}
{\;}
{0.25em}
{%
{\small\sffamily\color{genial}\thmname{#1}}%
{\nobreakspace\thmnumber{\@ifnotempty{#1}{}\@upn{#2}}}
\thmnote{{\nobreakspace\the\thm@notefont\sffamily\bfseries\color{black}\nobreakspace(#3)}} 
}

\newtheoremstyle{blacknumex}
{7pt}
{7pt}
{\normalfont}
{} 
{\small\bf\sffamily}
{\;}
{0.25em}
{%
{\small\sffamily\color{genial}\thmname{#1}}%
{\nobreakspace\thmnumber{\@ifnotempty{#1}{}\@upn{#2}}}
\thmnote{{\nobreakspace\the\thm@notefont\sffamily\bfseries\color{black}\nobreakspace(#3)}} 
}

\newtheoremstyle{blacknumbox} 
{7pt}
{7pt}
{\normalfont}
{}
{\small\bf\sffamily}
{\;}
{0.25em}
{%
{\small\sffamily\color{genial}\thmname{#1}}%
{\nobreakspace\thmnumber{\@ifnotempty{#1}{}\@upn{#2}}}
\thmnote{{\nobreakspace\the\thm@notefont\sffamily\bfseries\color{black}\nobreakspace(#3)}} 
}

\newtheoremstyle{genialnum}
{7pt}
{7pt}
{\normalfont}
{}
{\small\bf\sffamily\color{genial}}
{\;}
{0.25em}
{%
{\small\sffamily\color{genial}\thmname{#1}}%
{\nobreakspace\thmnumber{\@ifnotempty{#1}{}\@upn{#2}}}
\thmnote{{\nobreakspace\the\thm@notefont\sffamily\bfseries\color{black}\nobreakspace(#3)}} 
}

\makeatother

\RequirePackage[framemethod=default]{mdframed} 

\newmdenv[skipabove=7pt,
skipbelow=7pt,
rightline=false,
leftline=false,
topline=false,
bottomline=false,
backgroundcolor=black!5,
linecolor=genial,
innerleftmargin=5pt,
innerrightmargin=5pt,
innertopmargin=10pt,
leftmargin=0cm,
rightmargin=0cm,
innerbottommargin=10pt]{tBox}

\newmdenv[skipabove=7pt,
skipbelow=7pt,
rightline=false,
leftline=false,
topline=false,
bottomline=false,
backgroundcolor=genial!10,
linecolor=genial,
innerleftmargin=5pt,
innerrightmargin=5pt,
innertopmargin=5pt,
innerbottommargin=5pt,
leftmargin=0cm,
rightmargin=0cm,
linewidth=4pt]{eBox}	

\newmdenv[skipabove=7pt,
skipbelow=7pt,
rightline=false,
leftline=true,
topline=false,
bottomline=false,
linecolor=genial!50,
innerleftmargin=5pt,
innerrightmargin=5pt,
innertopmargin=5pt,
leftmargin=0cm,
rightmargin=0cm,
linewidth=4pt,
innerbottommargin=5pt]{dBox}	

\newmdenv[skipabove=7pt,
skipbelow=7pt,
rightline=false,
leftline=false,
topline=false,
bottomline=false,
linecolor=gray,
backgroundcolor=black!5,
innerleftmargin=5pt,
innerrightmargin=5pt,
innertopmargin=5pt,
leftmargin=0cm,
rightmargin=0cm,
linewidth=4pt,
innerbottommargin=5pt]{cBox}

\newmdenv[skipabove=7pt,
skipbelow=7pt,
rightline=false,
leftline=false,
topline=false,
bottomline=false,
linecolor=gray,
backgroundcolor=black!5,
innerleftmargin=5pt,
innerrightmargin=5pt,
innertopmargin=5pt,
leftmargin=0cm,
rightmargin=0cm,
linewidth=4pt,
innerbottommargin=5pt]{pBox}

\newmdenv[skipabove=7pt,
skipbelow=7pt,
rightline=false,
leftline=false,
topline=false,
bottomline=false,
linecolor=genialsol,
innerleftmargin=5pt,
innerrightmargin=5pt,
innertopmargin=0pt,
leftmargin=0cm,
rightmargin=0cm,
linewidth=4pt,
innerbottommargin=0pt]{solBox}	

\theoremstyle{genialnumbox}
\newtheorem{thm1}{Theorem}[section]
\newtheorem{ithm1}[thm1]{$\star$ THEOREM}
\newtheorem{ques1}[thm1]{Question}
\newtheorem{conj1}[thm1]{Conjecture}

\theoremstyle{blacknumex}
\newtheorem{exer}[thm1]{Exercise}
\newtheorem{exer*}[thm1]{$\ast$ Exercise}

\theoremstyle{blacknumbox}
\newtheorem{dfn1}[thm1]{Definition}

\theoremstyle{genialnum}
\newtheorem{cor1}[thm1]{Corollary}
\newtheorem{prop1}[thm1]{Proposition}
\newtheorem{lem1}[thm1]{Lemma}

\newtheorem{exm1}[thm1]{Example}

\newenvironment{theorem}{\paragraph{ } \begin{tBox}\begin{thm1}}{\end{thm1}\end{tBox}}

\newenvironment{exe*}{\paragraph{ } \begin{eBox}\begin{exer*}}{\hfill{\color{genial}
\ensuremath{\diamond\diamond\diamond}}\end{exer*}\end{eBox}}
	
\newenvironment{example}{\paragraph{ } \begin{exm1}}{\hfill{\tiny%
\ensuremath{\bigtriangleup\bigtriangledown\bigtriangleup}}\end{exm1}}
	
\newenvironment{ques}{\paragraph{ } \begin{cBox}\begin{ques1}}{\end{ques1}\end{cBox}}	
\newenvironment{conj}{\paragraph{ } \begin{cBox}\begin{conj1}}{\end{conj1}\end{cBox}}	

\newenvironment{proposition}{\paragraph{ } \begin{pBox}\begin{prop1}}{\end{prop1}\end{pBox}}	
\newenvironment{lemma}{\paragraph{ } \begin{pBox}\begin{lem1}}{\end{lem1}\end{pBox}}

\newenvironment{lem*}[1]{\vspace{1ex}\noindent
{\bf Lemma* (#1).} [restatement]  \hspace{0.5em} \em }{ }

\newenvironment{thm*}[1]{\begin{cBox}
\vspace{1ex}\noindent 
{\bf Theorem* (#1).} [restatement]  \hspace{0.5em} }{\end{cBox}}

\theoremstyle{genialnum}

\newtheorem*{clm*}{Claim}

\newenvironment{sol}%
{\begin{solBox}
\par \noindent 
\scriptsize
{\bf Solution to ex:{\color{blue} \arabic{exer}}.}  {\color{red} \ \  :( } \\ }%
{\hfill {\color{blue} :) $\checkmark$} \end{solBox}}

\newcommand{\ENDEXER}{
{\expandafter\comment}
{\expandafter\endcomment}
}

\makeatletter
\renewcommand{\@seccntformat}[1]{\llap{\textcolor{genial}{\csname the#1\endcsname}\hspace{1em}}}                    
\renewcommand{\section}{\@startsection{section}{1}{\z@}
{-4ex \@plus -1ex \@minus -.4ex}
{1ex \@plus.2ex }
{\normalfont\large\sffamily\bfseries}}
\renewcommand{\subsection}{\@startsection {subsection}{2}{\z@}
{-3ex \@plus -0.1ex \@minus -.4ex}
{0.5ex \@plus.2ex }
{\normalfont\sffamily\bfseries}}
\renewcommand{\subsubsection}{\@startsection {subsubsection}{3}{\z@}
{-2ex \@plus -0.1ex \@minus -.2ex}
{.2ex \@plus.2ex }
{\normalfont\small\sffamily\bfseries}}                        
\renewcommand\paragraph{\@startsection{paragraph}{4}{\z@}
{-2ex \@plus-.2ex \@minus .2ex}
{.1ex}
{\normalfont\small\sffamily\bfseries}}

\renewcommand\appendix{\par
     \newcounter{appsection}
     \setcounter{appsection}{\value{section}}
     \newcounter{valsection}
     \setcounter{valsection}{\value{section}}
     \addtocounter{valsection}{1-\value{appsection}}
  \setcounter{subsection}{0}%
  \setcounter{equation}{0}
  \gdef\thefigure{\@Alph\c@section.\arabic{figure}}%
  \gdef\thetable{\@Alph\c@section.\arabic{table}}%
  \gdef\thesection{
  \@Alph\c@valsection}%
  \@addtoreset{equation}{section}%
  \gdef\theequation{\@Alph\c@section.\arabic{equation}}%
}

\makeatother

\newcommand{\IP}[1]{\left\langle #1 \right\rangle}

\newcommand{\Integer}{\mathbb{Z}}

\newcommand{\Z}{\Integer}
\newcommand{\N}{\mathbb{N}}

\newcommand{\R}{\mathbb{R}}

\newcommand{\eps}{\varepsilon}

\newcommand{\ie}{{\em i.e.\ }}
\newcommand{\eg}{{\em e.g.\ }}

\def\squareforqed{\hbox{\rlap{$\sqcap$}$\sqcup$}}
\def\qed{\ifmmode\squareforqed\else{\unskip\nobreak\hfil
\penalty50\hskip1em\null\nobreak\hfil\squareforqed
\parfillskip=0pt\finalhyphendemerits=0\endgraf}\fi}

\newcommand{\ignore}[1]{ }

\newcommand{\p}{\partial}

\newcommand{\vphi}{\varphi}

\newcommand{\define}[1]{\textbf{#1}}

\newcommand{\stab}{\mathsf{stab}}

\renewcommand{\ker}{\mathsf{Ker}}

\title{Groups with finitely many Busemann points}
\author{Liran Ron-George}
\author{Ariel Yadin}
\address{Department of Mathematics, Ben-Gurion University of the Negev}
\email{ \{lirar, yadina\}@bgu.ac.il }
\thanks{We thank Y.\ Glasner, Y.\ Hartman, A.\ Karlsson, T.\ Meyerovitch, and M.\ Tointon for useful discussions. 
Research supported by 
Israel Science Foundation, 
grant no.\ 954/21.
The first author 
is also partially supported by the Israel Science Foundation, 
grant no.\ 1175/18.}

\begin{document}

\maketitle

\begin{abstract}
We show that an infinite finitely generated group $G$ is virtually $\Z$
if and only if every Cayley graph of $G$ contains only finitely many Busemann points in its horofunction boundary.
This complements a previous result of the second named author and M.\ Tointon.
\end{abstract}

\section{Introduction}

In a few important works it has been suggested to consider horofunctions as an analogue to linear functionals,
in (non-linear) general metric spaces, see \eg \cite{Karlsson2002nonexpanding, Karlsson21linear, Karlsson21, Yau06}
and references therein.  The space of horofunctions (which will be precisely defined below) has been fruitfully used
in the study of the geometry of metric spaces.  Specifically in the case of hyperbolic spaces and hyperbolic groups 
these have been used quite successfully.  Less research has been devoted to the study of horofunction boundaries in 
the case of ``small'' groups.  This is perhaps due to the fact that the space of horofunctions itself is not invariant 
when switching between different Cayley graphs of the same group.  However, it seems to be true that some properties 
are shared by all Cayley graphs of a given group, and it is very interesting to determine such invariants.  

Horofunctions are very well suited to finding {\em virtual characters}, that is functions whose restriction to a finite 
index subgroup is a non-trivial homomorphism into an Abelian group.  In fact, as will be explained below, 
if a horofunction has a finite orbit under the canonical group action, it is such a virtual character (with the implicit 
finite index subgroup being the stabilizer of the action).
Finding such virtual characters is useful in many situations, because it provides a way of splitting the group into two 
parts.  Specifically, if $\psi : G \to \Z$ is a surjective homomorphism, then $G \cong \Z \ltimes K$ for $K = \ker \psi$.
Even though such a decomposition is a semi-diect product, in many cases it is still useful in the 
geometric or algebraic analysis.  For example, the main part in the proof of
Gromov's theorem regarding groups of polynomial growth \cite{Gromov81}
is to show that every group of polynomial growth admits a virtual character.  
Most of the new proofs of this theorem are actually proofs of this result, including Kleiner's proof \cite{Kleiner10}
and Ozawa's proof \cite{Ozawa18}.
Even before these new proofs appeared it was suggested by Karlsson \cite{Karlsson08} to use the 
horofunction boundary to provide the required virtual character in the polynomial growth case.

\begin{conj} \label{conj: poly growth implies countable boundary}
Let $G$ be a finitely generated group of polynomial growth.
Then, for any Cayley graph of $G$ the horofunction boundary contains a finite orbit
(for the canonical group action).
\end{conj}

See also Conjecture 1.3 in \cite{TY16}.

Grigorchuk has conjectured that there is a ``gap'' in the possible growth functions of Cayley graphs,
see \cite{Grigorchuk90}.  Specifically, if the growth is small enough it is conjectured that the group
actually has polynomial growth.  Since horofunctions, and specifically Busemann points (see below) 
are related to geodesics in the graph, these may very well be good tools to study the growth.
By Gromov's theorem mentioned, every polynomial growth group is virtually nilpotent and 
thus admits a virtual character.  
Walsh \cite{Walsh11} has shown that the action of nilpotent groups on their horofunction boundary has 
a finite orbit.
So we have the following (logically weaker) conjecture.

\begin{conj}
\label{conj:weak gap}
Let $\Gamma$ be a Cayley graph.  Assume that the size of the ball of radius $r$ in $\Gamma$ 
is bounded by $C \exp ( C r^\alpha)$ for some $C>0$, $\alpha < \tfrac12$ and all $r \in \N$.

Then, the horofunction boundary of $\Gamma$ contains a finite orbit for the underlying group action.
\end{conj}

Let us remark that even smaller growth functions in Conjecture \ref{conj:weak gap} are of interest.
The state of the art follows from the work of Shalom \& Tao \cite{ShalomTao} proving that any group
of growth bounded by $C \exp ( C (\log r) (\log \log r)^\eps)$ for some $\eps>0$ is virtually nilpotent (and so
actually of polynomial growth).
In forthcoming work by Szusterman and the second named author it is shown that if every Cayley graph of small enough 
growth has a finite orbit in its horofunction boundary, then the underlying group is actually virtually nilpotent.

Part of the difficulty of understanding the horofunction boundary of a Cayley graph,
and especially connecting it to the underlying algebraic structure of the group, is that the boundary changes 
when changing the specific choice of Cayley graph.  
The properties of the boundary which are invariant to changing Cayley graphs are still somewhat mysterious.
As a first step in considering such invariants, in \cite{TY16} it was shown that any graph that has linear volume growth
must have finitely many Busemann points in its boundary (see below for a precise definition of Busemann points).
The other direction is not true for general graphs, 
we provide an example of a graph with arbitrary growth and a boundary 
consisting of just one point, see Example \ref{exm:general graph exp growth}
below.

Our main result, Theorem \ref{thm:main}, is a first step in understanding the relationship between 
algebraic properties of a group and the horofunction boundaries of its Cayley graphs.
Theorem \ref{thm:main} states that an infinite finitely generated group is virtually $\Z$ if and only if 
for every Cayley graph there are finitely many Busemann points in its horofunction boundary.
It is also shown that these conditions are equivalent to every Cayley graph
having a finite horofunction boundary, and that when these equivalent conditions hold, 
every horofunction is a Busemann point.
See below for a precise statement and definitions.

\subsection{Notation and precise statement of results} 

%
%
%
%
%

Let $G$ be a finitely generated group.
Fix a finite symmetric generating set $S$ for $G$; that is, $S$ generates $G$ as a group, $|S| < \infty$, and 
$S = S^{-1} = \{ x^{-1} \ : \ x \in S \}$.
The \define{Cayley graph} of $G$ with respect to $S$, denoted $\Gamma(G,S)$ is the graph
whose vertices are elements of $G$ and edges are given by the relation: $\{x, xs\}$ is an edge for any $x \in G$ 
and $s \in S$.  The graph metric of $\Gamma(G,S)$ is denoted by $d_S$.
It is easy to check that $d_S$ is left-invariant in the sense that $d_S(zx,zy) = d_S(x,y)$ for all $x,y,z \in G$.
Thus it is useful to use the notation $|x|_S = d_S(x,1)$ where $1=1_G$ denote the identity element in the group $G$.
When the metric is clear from the context we will omit the subscript $d = d_S$ and $|x| = |x|_S$. 

Fix some Cayley graph $\Gamma(G,S)$.
Let $L(G,S)$ be the set of all functions $h : G \to \R$ such that $h$ is 
$1$-Lipschitz (\ie $|h(x) - h(y)| \leq d(x,y)$ for all $x,y \in G$) and $h(1) = 0$.
Equip $L(G,S)$ with the topology of pointwise convergence, and note that $L(G,S)$ is compact.
The set $G$ embeds into $L(G,S)$ by identifying $x \in G$ 
with the {\em Busemann function} $b_x : G \to \R$ given by $b_x(y) = d(x,y) - d(x,1)$.
We denote closure of $\{ b_x \ : \  x \in G\}$ in $L(G,S)$ by $\overline{ \Gamma(G,S) }$ and define 
the \define{horofunction boundary}, or \define{horoboundary} of $\Gamma(G,S)$ to be 
$$ \p \Gamma(G,S) : = \overline{\Gamma(G,S)} \setminus \{ b_x \ : \ x \in G \} . $$
Elements of $\p \Gamma(G,S)$ are called \define{horofunctions}.
(The term {\em metric-functional} is also used, with good reason, see \cite{Karlsson22}.)
One notes that all horofunctions are integer valued (because the Busemann functions 
are, and we are dealing with pointwise convergence).

A \define{finite geodesic} in $\Gamma(G,S)$ is a finite sequence $(\gamma_k)_{k=0}^n$
such that 
\begin{align}
\label{eqn:geodesic dfn}
\forall \ 0 \leq m \leq k \leq n & \qquad d(\gamma_k , \gamma_m) = k-m .
\end{align}
An \define{infinite geodesic} in $\Gamma(G,S)$ is an infinite sequence $(\gamma_k)_{k=0}^\infty$ such that
\eqref{eqn:geodesic dfn} holds for all $0 \leq m \leq k$.
For simplicity, when we write {\em geodesic} we refer to an infinite geodesic, and if we require finite geodesics 
we will explicitly state {\em finite}.

It is shown in Lemma \ref{lem:geodesics} below that
if $(\gamma_n)_n$ is an infinite geodesic, then $\gamma_\infty : = \lim_{n \to \infty} b_{\gamma_n}$ exists 
(as a limit in $\overline{ \Gamma (G,S) }$), and that $\gamma_\infty \in \p \Gamma (G,S)$.
Such a horofunction $\gamma_\infty \in \p \Gamma(G,S)$ that is a limit of a geodesic is called a \define{Busemann point}.
We denote by $\p_b \Gamma (G,S)$ the closure (in $\overline{\Gamma(G,S)}$) of the set of all Busemann points.
For a more in depth discussion of Busemann points on Cayley graphs see \cite{WW06} and references therein.

The fact that any Cayley graph contains an infinite geodesic is classical, and can be 
proven using K\"onig's Lemma. (In fact, that proof shows the existence of a {\em bi-infinite} geodesic,
\ie a sequence $(\gamma_t)_{t \in \Z}$ such that $\gamma_0=1$ and such that any finite
subpath is a finite geodesic.  This implies that there are always at least two Busemann points in the 
horoboundary.
See also Karlsson's proof (by another method) in \cite{Karlsson21} 
that the horoboundary (there called metric boundary) of a finitely generated infinite group 
contains at least $2$ elements.)

The horoboundary construction also provides an equivalence relation on geodesics:
two geodesics $\alpha, \beta$ are equivalent if $\alpha_\infty = \beta_\infty$.
Note that if $\alpha \cap \beta := \{ x \ : \ \exists \ t,k \in \N \ , \ \alpha_t = \beta_k = x \}$ is infinite,
then $\alpha_\infty = \beta_\infty$.  However, the opposite is not guarantied.  One may have 
equivalent geodesics which have an empty intersection.
(It is however true that geodesics $\alpha ,\beta$ are equivalent if and only 
if there exists some geodesic $\gamma$ such that $|\gamma \cap \alpha| = |\gamma \cap \beta | = \infty$.
See \cite{Walsh11}.)

Our main result is the following.

\begin{theorem}
\label{thm:main}
Let $G$ be an infinite finitely generated group.  
The following are equivalent.
\begin{enumerate}
\item $G$ is virtually $\Z$.
\item For every Cayley graph of $G$ the Busemann boundary is finite.
\item For every Cayley graph of $G$ the horofunction boundary is finite.
\end{enumerate}
Moreover, when $G$ is virtually $\Z$ any horofunction is a Busemann point.
\end{theorem}

Recall that for a group property $\mathcal{P}$, the property of being virtually $\mathcal{P}$
is containing a finite index subgroup with the property $\mathcal{P}$.
Thus, virtually $\Z$ means containing a finite index infinite cyclic subgroup.

\begin{proof}[Proof of Theorem \ref{thm:main}]
$(1) \Rightarrow (2)$  was first shown in \cite{TY16},
and holds for general graphs of linear growth.
In Proposition \ref{prop:linear growth implies finite Busemann} we provide a very elegant and short proof by Sam Shepperd (communicated to us by M.\ Tointon).

Our main two new contributions are $(2) \Rightarrow (1)$
and the assertion that if $G$ is virtually $\Z$ then every horofunction
is a Busemann point.

The implication $(2) \Rightarrow (1)$ is shown in Lemma \ref{lem:main lem} below.

Theorem \ref{thm:horofunction are Busemann} below shows that when $G$ is virtually 
$\Z$ then in any Cayley graph, every horofunction is a Busemann point.
This gives $(1) \Rightarrow (3)$ by combining the implication $(1) \Rightarrow (2)$
with Theorem \ref{thm:horofunction are Busemann}.

$(3) \Rightarrow (2)$ is trivial since the Busemann boundary is contained in the 
horoboundary.
\end{proof}

Somewhat frustratingly, we have not proved the following.

\begin{conj} \label{conj:exists a finite Busemann boundary}
Let $G$ be an infinite finitely generated group and assume that there exists some Cayley graph of $G$ for which $\p_b \Gamma(G,S)$ is finite.  

Then, 
$G$ is virtually $\Z$.
\end{conj}

\section{Group Action}

There is a natural action of $G$ on $\overline{\Gamma(G,S)}$,
given by $x.h(y) = h(x^{-1}y) - h(x^{-1})$ for $x \in G$ and $h \in \overline{\Gamma(G,S)}$.
The following basic properties are very easy to verify, and we leave the proof to the reader.

\begin{proposition}
\label{prop:action}
The action of $G$ on $\overline{\Gamma(G,S)}$ has the following properties:
\begin{enumerate}
    \item The action is continuous, \ie if $h_n \to h$ then $x.h_n \to x.h$.
    \item
    The map $x \mapsto b_x$ is equivariant, \ie $x.b_y = b_{xy}$. 
    \item
    The horoboundary $\partial \Gamma(G,S)$ and the Busemann boundary $\partial_b \Gamma(G,S)$ are invariant subsets.
\end{enumerate}
\end{proposition}

Two basic properties of horofunctions we require are the following.
The proofs are included for completeness.

\begin{proposition}
\label{prop:unbounded horofunctions}
Let $\Gamma(G,S)$ be a Cayley graph.

If $b_{x_n} \to h$ and $h \in \p \Gamma(G,S)$, 
then $|x_n| \to \infty$.

Also, if $b_{x_n} \to h$ and $|x_n| \to \infty$,
then for every $r \geq 0$ there exists $x \in G$ with $h(x) = - |x| = -r$.

In particular, horofunctions are not bounded from below.
\end{proposition}

\begin{proof}
Assume that $b_{x_n} \to h \in \overline{ \Gamma(G,S)}$ 
and $|x_n| \leq r$ for all $n$.
Then, since the ball of radius $r$ is finite, it must be
that there exists some $|x| \leq r$ such that $x_n = x$
for infinitely many $n$.  Thus, $h = b_x \not\in \p \Gamma(G,S)$.
This shows that if $b_{x_n} \to h \in \p \Gamma(G,S)$ then 
$|x_n| \to \infty$.

Now, assume that $b_{x_n} \to h$ with $|x_n| \to \infty$.
Fix some $r \geq 0$, let $z = x_n \in G$ be such that 
$b_z(y) = h(y)$ for all $|y| \leq r$.  
Since $|x_n| \to \infty$ we can choose $z$ so that $|z| > r$.
Let $\gamma = (1=\gamma_0, \ldots, \gamma_{|z|} = z)$ be a finite  
geodesic and let $x = \gamma_r$.  Note that $|x|=r$ and that
$$ h(x) = d(\gamma_{|z|} , \gamma_r) - d(\gamma_{|z|} , \gamma_0) = - r = - |x| . $$
\end{proof}

\begin{proposition} \label{prop:stabilizer}
Let $h \in \p \Gamma(G,S)$ be a horofunction.
Let $K = \stab(h) = \{ x \in G \ : \ x.h=h \}$ be the stabilizer of $h$.

Then, the restriction $h \big|_K$ is a
homomorphism from $K$ into $\Z$ (the additive group of integers).
\end{proposition}

\begin{proof}
This follows immediately since for any $x \in K$ and any $y \in G$ we have
$$ h(x y) = x^{-1}.h(y) + h(x) = h(y) + h(x) . $$
\end{proof}

We have seen that restricting a horofunction to the stabilizer subgroup 
(for the canonical group action) results in a homomorphism into $\Z$.
One naive approach to attempt to provide a simple proof that a finite horoboundary
implies that the group is virtually $\Z$ 
(similar to Conjecture \ref{conj:exists a finite Busemann boundary}) 
would be as follows.

$G$ acts on $\p \Gamma(G,S)$.  
Define:
$$ K= \{ x \in G \ : \ \forall \ h \in \p \Gamma(G,S) \ , \ x.h=h \} $$
$$ F = \{ x \in  K \ : \ \forall \ h \in \p \Gamma(G,S) \ , \ h(x) =0 \} . $$
One can show that $F$ is a subgroup.
In fact it is possible to show that if $x_1, \ldots, x_n \in F$
then the group they generate, $\IP{ x_1, \ldots, x_n }$,
is contained in a ball of finite radius $r$, which depends only on
the generators $x_1, \ldots, x_n$.

At the moment we do not know an answer to the following:
\begin{ques}
Is the subgroup $F$ finitely generated when the horoboundary is finite? 
Is $F$ always finitely generated (in any Cayley graph)?
\end{ques}

Because of the properties above, if $F$ is finitely generated, then it is finite.

When the horoboundary is finite, we have that $[G:K] < \infty$ 
and we can also write 
$\p \Gamma (G,S) = \{ h_1, \ldots, h_d \}$.
The map $\Psi (x) = (h_1(x), \ldots, h_d(x))$ is then a map into $\Z^d$
whose restriction to $K$ is a homomorphism.  The kernel of $\Psi \big|_K$
is exactly $F$.  So if $F$ is finite, then one could show that $K$ is virtually
$\Z^n$ for some $n \leq d$, implying that also $G$ is.  
Since $G$ has a finite horoboundary, one can then 
prove that actually $n=1$, along the lines of our proof of Theorem \ref{thm:main}.

\section{Finite Busemann boundaries}

In this section we prove the main new part of Theorem \ref{thm:main},
stated as the following lemma.

\begin{lemma} \label{lem:main lem}
Let $G$ be a finitely generated infinite group.  Assume that $S$ is a finite symmetric generating set 
such that for any $x \in G$ we have that the Busemann boundary $\p_b \Gamma(G, U)$
is finite, where $U = S \cup \{x,x^{-1}\}$.  

Then, $G$ is virtually $\Z$.
\end{lemma}

The proof of Lemma \ref{lem:main lem} is at the end of the section, and will use the following lemmas.

\begin{lemma} \label{lem:geodesics}
Let $\Gamma(G,S)$ be a Cayley graph.
Let $(\gamma_t)_t$ be a geodesic.

Then, the pointwise limit 
$$ \gamma_\infty = \lim_{t \to \infty} b_{\gamma_t}  $$
exists, and $\gamma_\infty \in \p \Gamma(G,S)$. 

Then, there exists $t_0$ such that for all $t \geq t_0$ we have that $\gamma_\infty(\gamma_t) = - |\gamma_t|$.

If $\gamma_0 = 1$, then $\gamma_\infty(\gamma_t) = -t$ for all $t \in \N$.
\end{lemma}

\begin{proof}
Let $(\gamma_t)_t$ be a geodesic with $\gamma_0=1$.
Set $b_t = b_{\gamma_t}$.  For any $x \in G$, by the triangle inequality 
$b_x(y) = d(y,x) - |x| \geq - |y|$.
Thus, for fixed $y \in G$ the sequence $(b_t(y))_t$ is bounded below by $-|y|$.
Because $\gamma$ is a geodesic with $\gamma_0=1$, 
we have that $|\gamma_{t+1}| - |\gamma_t| = d(\gamma_{t+1},\gamma_t)$.
So another use of the triangle inequality gives,
\begin{align*}
b_{t+1} (y) - b_t(y) & = d(\gamma_{t+1}, y) - d(\gamma_t,y) - d(\gamma_{t+1},\gamma_t) \leq 0 .
\end{align*}
We conclude that $(b_t(y))_t$ is a non-increasing sequence of integers, bounded from below, and hence 
must converge to some integer which we denote by $\gamma_\infty(y)$.

Also, for every $s \geq t$ we have that
$$ b_s(\gamma_t)  = d(\gamma_t,\gamma_s) - |\gamma_s| = - |\gamma_t| . $$
Taking $s \to \infty$ we have that $\gamma_\infty(\gamma_t) = -|\gamma_t| = - t$ for all $t$.

Now, let $\beta$ be a geodesic, with $\beta_0 = x$.
Then, $(\gamma_t = x^{-1} \beta_t)_t$ is a geodesic with $\gamma_0=1$,
so the limit $\beta_\infty = x.\gamma_\infty$ exists.

Let $t_0$ be such that for all $t \geq t_0$ we have that 
$\gamma_\infty(x^{-1}) = d(\gamma_t, x^{-1}) - |\gamma_t|$.
Then, for all $ s \geq t \geq t_0$ we have that 
$$ d(\gamma_s,\gamma_t) - (d(\gamma_s,x^{-1}) - d(\gamma_t,x^{-1})) = 
d(\gamma_t,x^{-1}) - |\gamma_t| - ( d(\gamma_s,x^{-1}) - |\gamma_s|) = 0 , $$
so that for all $t \geq t_0$ and all large enough $s$,
\begin{align*}
\beta_\infty(\beta_t) & = d(\beta_s, \beta_t) - |\beta_s| \\
& = d(\gamma_s,\gamma_t) - (d(\gamma_s,x^{-1}) - d(\gamma_t,x^{-1})) - |\beta_t| 
= - |\beta_t| .
\end{align*}
\end{proof}

\begin{lemma}
\label{lemma:x for finite index}
Let $\Gamma(G,S)$ be a Cayley graph.
Let $K \leq G$ be a finite index subgroup $[G:K]<\infty$. 

Then, there exists $h \in \partial_b \Gamma(G,S)$ and $1 \neq y \in K$ such that $h(y)=-|y|$.
\end{lemma}

\begin{proof}
Let $R \subset G$ be finite set of representatives for the cosets of $K$, \ie $G=\uplus_{r \in R} rK$, 
and suppose that $1 \in R$. 
Let $(\gamma_t)_t$ be some geodesic in $\Gamma(G,S)$ such that $\gamma_0=1$ and let 
$\gamma_{\infty} \in \partial_b \Gamma(G,S)$. 
By Lemma \ref{lem:geodesics} we have that $\gamma_{\infty}(\gamma_t)=-t$ for every $t$. 
Since $R$ is finite, there exists $r \in R$ such that $\gamma_t \in rK$ for infinitely many $t$. 
Denote $\beta_t=r^{-1}\gamma_t$. 
Note that $\beta_t$ is a geodesic and $\beta_{\infty}=r^{-1}.\gamma_{\infty}$. 
By Lemma \ref{lem:geodesics}
there exists $t_0$ such that for every $t \geq t_0$ we have 
$\beta_{\infty}(\beta_t)=-|\beta_t|$. 
Since by definition $\beta_t \in K$ for infinitely many $t$, there exists $t$ such that $\beta_t \in K$ and 
$\beta_{\infty}(\beta_t)=-|\beta_t|$. 
We are done by taking $h=\beta_{\infty}$ and $y=\beta_t$.
\end{proof}

\begin{lemma}
\label{lemma:map and geodesics}
Let $\Gamma(G,S)$ be a Cayley graph.
Let $K \leq G$ be a subgroup. 

Suppose there exists a $1$-Lipschitz map $h:G \to \R$ 
such that $h|_K$ is a group homomorphism, 
and suppose also that there is $1 \neq x \in K$ such that $h(x)=|x|_S$.

Then $|x^t|_S=t \cdot |x|_S$ for every $t \in \mathbb{N}$. 
If we denote $U=S \cup \{x,x^{-1}\}$ then for any $g \in G$ we have that 
$(gx^t)_t$ is a geodesic in $\Gamma(G,U)$.

Moreover, there exists a geodesic $\gamma$ in $\Gamma(G,S)$ such that 
$\gamma_{t|x|} = x^t$
for all $t \in \N$.
\end{lemma}

\begin{proof}
By the triangle inequality $|x^t|_S \leq t |x|_S$.
By our assumptions on $h$, since $x \in K$ we have that
$$ t |x|_S = t h(x) = h(x^t) \leq |x^t|_S , $$
so $|x^t|_S = t |x|_S = h(x^t)$ for all $t \in \N$.

Next we show that $|x^t|_U=t$ for every $t \in \N$. 
Write $x^t=u_1 \cdots  u_m$ where $u_j \in U$ for all $1 \leq j \leq m = |x^t|_U$. 
Let $J=\{j \ | \ u_j \in S\}$, so that if $j \notin J$ then $u_j \in \{x,x^{-1}\}$. 
Since $x \in U$, $m \leq t$. Also,
$$ t \cdot |x|_S=|x^t|_S = | u_1  \cdots u_m|_S 
\leq \sum_{j=1}^m |u_j|_S \leq |J|+(m-|J|) \cdot |x|_S \leq m \cdot |x|_S , $$
where we have used that $|x|_S \geq 1$ because $x \neq 1$. Thus, $t \leq m$, implying that $|x^t|_U = m = t$.

Now, let $g \in G$. Since $(x^t)_t$ is a geodesic in $\Gamma(G,U)$ and the graph metric is left invariant, $d_U(gx^n,gx^m)=d_U(x^n,x^m)=|n-m|$, so $(gx^t)_t$ is also a geodesic in $\Gamma(G,U)$.

Finally, for the last assertion,
we want to construct a geodesic $\gamma$
in $\Gamma(G,S)$ such that $\gamma_{t|x|} = x^t$
for all $t \in \N$.
To this end, set $m=|x|$ and let 
$(1=x_0, x_1, \ldots, x_m=x)$ be a finite geodesic in $\Gamma(G,S)$.
Define $\gamma_0=1$ and for $0 \leq j \leq m$ define $\gamma_{nm+j} = x^n x_j$.

We will show that $\gamma$ is a geodesic in $\Gamma(G,S)$.
Indeed, since $\gamma$ is a path in the graph $\Gamma(G,S)$, it is immediate
that $d(\gamma_{t} , \gamma_{s}) \leq s-t$ for all $s \geq t$,
so we only need to prove a matching lower bound.

Let $k \geq n$ and $0 \leq i , j \leq m$ be such that $km+i \geq nm+j$.
Then,
\begin{align*}
d(\gamma_{nm+j} , \gamma_{km+i}) & = d(x^n x_j , x^k x_i) = d(x_j , x^{k-n} x_i) .
\end{align*}
If $k=n$ then $i \geq j$ and 
$$ d(\gamma_{nm+j} , \gamma_{km+i}) = d(x_j,x_i) = i-j = km+i - (nm+j) $$
because $(x_0, \ldots, x_m)$ is a finite geodesic.
If $k >n$, then
\begin{align*}
(k-n+1)|x|_S & = |x^{k-n+1}|_S \leq d(x^{k-n+1},  x^{k-n} x_i ) + d(x^{k-n} x_i , x_j)
+ d(x_j,1) \\
& = d(x_m , x_i) + d(\gamma_{nm+j} , \gamma_{km+i}) + d(x_j,x_0) \\
& = d(\gamma_{nm+j} , \gamma_{km+i}) + |x|_S + j-i ,
\end{align*}
where we have used again that $(x_0,\ldots,x_m)$ is a finite geodesic.
Thus, in all cases we get that
$$ d(\gamma_{nm+j} , \gamma_{km+i}) \geq km+i - (nm+j) , $$
providing a matching lower bound to the upper bound, and proving that 
$\gamma$ is indeed a geodesic in $\Gamma(G,S)$.
\end{proof}

\begin{lemma}
\label{lemma:at least ker(h) busemann}
Let $\Gamma(G,S)$ be a Cayley graph.
Let $K \leq G$ be a subgroup. 

Suppose there exists a $1$-Lipschtiz map $h:G \to \R$ 
such that $h|_K$ is a group homomorphism, 
and suppose also that there is $1 \neq x \in K$ such that $h(x)=|x|_S$.
Let $U=S \cup \{x,x^{-1}\}$. 

Then, 
$$ |\ker{(h|_K)}| \leq |\partial_b \Gamma(G,U)| . $$
\end{lemma}

\begin{proof}

Let $g \in \ker(h|_K)$.
Let $\gamma_t = x^t$ and $\beta_t = g x^t$.  Note that both $(\gamma_t)_t$ and $(\beta_t)_t$ 
are geodesics in $\Gamma(G,U)$ by Lemma \ref{lemma:map and geodesics}.

Define $\psi(x)=\frac{1}{|x|_S} \cdot h(x)$. Notice that $\psi$ is a $1$-Lipschitz map on $\Gamma(G,U)$, 
since $|\psi(s)| \leq 1$ for every $s \in S$ and $|\psi(x)|=|\psi(x^{-1})|=1$. 
Note also that $\psi|_K$ is a homomorphism to $\R$, and satisfies $\psi(\gamma_t)=\psi(\beta_t)=t$ for every $t$. 
Thus for every $t$,
$$d_U(\gamma_t,\gamma_0)=t=|\psi(\gamma_t)-\psi(\beta_0)| \leq d_U(\gamma_t,\beta_0)$$
$$d_U(\beta_t,\beta_0)=t=|\psi(\beta_t)-\psi(\gamma_0)| \leq d_U(\beta_t,\gamma_0)$$
Consider $\overline{\Gamma(G,U)}$. Note that,
\begin{align}
\label{eqn:gamma and beta}
b_{\gamma_t}(\gamma_0)-b_{\gamma_t}(\beta_0) & =d_U(\gamma_t,\gamma_0)-d_U(\gamma_t,\beta_0) \leq 0 
\nonumber \\
& \leq d_U(\beta_t,\gamma_0)-d_U(\beta_t,\beta_0)=b_{\beta_t}(\gamma_0)-b_{\beta_t}(\beta_0) .
\end{align}

If $\gamma_{\infty}=\beta_{\infty}$, there exists $t_0$ such that for every $t \geq t_0$ we have that $b_{\gamma_t}(\gamma_0)-b_{\gamma_t}(\beta_0)=b_{\beta_t}(\gamma_0)-b_{\beta_t}(\beta_0)$, implying equality throughout
\eqref{eqn:gamma and beta}.
Thus, $d_U(\gamma_t,\beta_0)=d(\gamma_t,\gamma_0)=t$ for every $t \geq t_0$.

Now we shall prove that $g=1$. Since $h(g)=0$, $h(x^n)=n \cdot |x|_S$ and $h$ is $1$-Lipschitz, we have 
$$ t \cdot |x|_S=h(\gamma_t)=|h(\gamma_t)-h(g)| \leq d_S(\gamma_t,g) , $$ 
so that $t \cdot |x|_S \leq d_S(\gamma_t,g)$ for every $t$. 

Fix some $t \geq t_0$. 
Since $d_U(\gamma_t,\beta_0)=t$ there exist a finite geodesic 
$(z_j)_{j=0}^t$ in $\Gamma(G,U)$ such that $z_0=\beta_0=g$ 
and $z_t=\gamma_t=x^t$. 
Let $u_j=z_{j-1}^{-1}z_j$ for every $1 \leq j \leq t$, and let $J=\{j ~|~ u_j \in S\}$. We get:
$$t \cdot |x|_S \leq d_S(\gamma_t,g) \leq \sum_{j=1}^t |u_j|_S=|J|+(t-|J|) \cdot |x|_S=t \cdot |x|_S-|J| \cdot (|x|_S-1) . $$
Because $x \neq 1$, it follows that $|J|=0$, 
so $u_j \in \{x,x^{-1}\}$ for every $1 \leq j \leq t$. 
Thus $x^t=z_t=gx^m$ for some $m \in \Z$. 
But that means that $g=x^{t-m}$ and so $0=h(g)=h(x^{t-m})=(t-m) \cdot |x|_S$. 
This implies that $t-m=0$, \ie that $g=1$.

Thus we have shown that the only $g \in \ker(h|_K)$ such that $\gamma_\infty = g.\gamma_\infty$ is $g=1$.
So if $g,g' \in \ker(h|_K)$ are such that $g.\gamma_\infty = g'.\gamma_\infty$ then $g=g'$.
That is, the map $g \mapsto g. \gamma_\infty$ from $\ker(h|_K)$ to $\p_b \Gamma(G,U)$ is injective, completing the proof.
\end{proof}

We now complete this section by proving Lemma \ref{lem:main lem}.

\begin{proof}[Proof of Lemma \ref{lem:main lem}]
Let $S$ be a finite symmetric generating set for an infinite group $G$ such that 
for any $x \in G$, setting $U = S \cup \{ x ,x^{-1}\}$ gives a Cayley graph $\Gamma(G,U)$ with a 
finite Busemann boundary.  
Specifically, taking $x \in S$ tells us that $|\p_b \Gamma(G,S)| < \infty$.
Let 
$$ K = \{ x \in G \ : \ \forall \ h \in \p_b \Gamma(G,S) \ , \ x.h=h \} $$
be the kernel of the action.  Since $\p_b \Gamma(G,S)$ is finite, $[G:K] < \infty$.
By Lemma \ref{lemma:x for finite index} there exists $h \in \p_b \Gamma(G,S)$ and $1 \neq y \in K$ such that 
$h(y) = -|y|_S$.  By Proposition \ref{prop:stabilizer}, $h|_K$ is a homomorphism into $\R$, 
so $h(x) = |x|_S$ for $x = y^{-1}$.
By Lemma \ref{lemma:at least ker(h) busemann}, we have that $| \ker(h|_K) | \leq | \p_b \Gamma(G,U) |$,
where $U = S \cup \{x,x^{-1} \}$.  This is finite by our assumptions on $S$.

Since $h$ is a non-trivial integer valued homomorphism on the infinite group $K$, 
it must be that $h(K) \cong \Z$. 
Thus, $K / \ker(h|_K) \cong \Z$, and as the kernel is finite, this implies that $K$ is virtually $\Z$.
As $[G:K] < \infty$, this completes the proof.
\end{proof}

\section{Graphs of Linear Growth}

The ``only if'' direction of Theorem \ref{thm:main} was originally proven in \cite{TY16}
(in that paper horoboundary refers to what we call Busemann boundary).
Sam Shepperd gave the following short and elegant proof.

By a graph of linear growth we mean that the number of vertices in the ball of radius 
$r$ in the graph grows at most linearly in the radius $r$.

\begin{proposition}
\label{prop:linear growth implies finite Busemann}
Let $\Gamma$ be a graph, with $d$ denoteing the graph distance. 
Let $o \in \Gamma$ be some vertex, and consider the Busemann boundary $\p_b \Gamma$
with respect to this base point.  
Let $S_r = \{ x \in \Gamma \ : \ d(x,o) = r \}$ be the sphere of radius $r$ around $o$.

Then, 
$$ | \p_b \Gamma | \leq \liminf_{r \to \infty} |S_r| . $$

Specifically, if $\Gamma$ has linear growth, then the Busemann boundary is finite.
\end{proposition}

\begin{proof}
For a geodesic $\gamma$ in $\Gamma$, denote  
$$ S_r \cap \gamma = \{ x \in S_r \ : \ \exists \ t \ , \ \gamma_t = x \} . $$
For any geodesic $\gamma$ there exists some $r(\gamma)$ such that 
for all $r \geq r(\gamma)$ we have $|S_r \cap \gamma| = 1$.

Now,
let $\gamma^{(1)} , \ldots, \gamma^{(n)}$ be $n$ geodesics, 
such that $\gamma^{(1)}_\infty , \ldots, \gamma^{(n)}_\infty$ are all distinct Busemann points.
Since these geodesics are all pairwise non-equivalent, no two of them can intersect infinitely many times.  
Thus, there exists some $r_0$ such that for all $r \geq r_0$, 
we have
$$ \forall \ i \neq j \ , \  \big( S_r \cap \gamma^{(i)} \big) \cap \big( S_r \cap \gamma^{(j)} \big) = \emptyset . $$
By making sure that $r_0 \geq \max\{ r(\gamma^{(1)}) , \ldots, r(\gamma^{(n)}) \}$ we get that
$$ n = \big| \bigcup_{j=1}^n (S_r \cap \gamma^{(j)})  \big| \leq |S_r|  $$
for all $r \geq r_0$.
Taking $\liminf$ on the right-hand side completes the proof of the first assertion.

Now, if $\Gamma$ has linear growth, then there exists a constant $C>0$
such that for any $r \in \N$ we have 
$$ \big| \{ x \in \Gamma \ : \ d(x,o) \leq r \} \big| =  
\sum_{k=0}^r |S_k| \leq C (r+1) . $$
This is easily seen to imply that $\liminf_{r \to \infty} |S_r| < \infty$, 
implying the second assertion.
\end{proof}

Note also that one can replace the geodesics in the proof above 
with infinite simple paths which have distinct limits in the horoboundary.
It is not true, however, that every horofunction is the limit of a sequence forming a simple infinite path.

As mentioned, the converse statement to Proposition \ref{prop:linear growth implies finite Busemann}
for general graphs, \ie an analogue of Theorem \ref{thm:main}, is not true in general.
Consider the following example with a graph of arbitrary growth, but only one point in 
the horofunction boundary.

\begin{example}
\label{exm:general graph exp growth}
Consider the following graph:
Let $(\Gamma_n)_{n \in \N}$ be a sequence of finite graphs, 
and fix some vertex $x_n \in \Gamma_n$ in each one.

The vertex set of our graph $\Gamma$ is then defined to be $\N \cup \bigcup_n \Gamma_n$.
Edges in $\Gamma$ are given by: 
\begin{itemize}
\item the original edges in $\N$ ($\{x,y\}$ is an edge if $|x-y|=1$
for $x,y \in \N$),  
\item the original edges in each $\Gamma_n$, 
\item and additional edges $\{n,x_n\}$ for each $n \in \N$.
\end{itemize}
See Figure \ref{fig:grove}.

It is not difficult to compute that if $x \in \Gamma_n \cup \{n\}$
and if $y \in \Gamma_m \cup \{m\}$ for $m>n$, then
$$ d(y,x) = d(y,m) + d(m,n) + d(n,x) , $$
so that the only possible horofunction in this graph is given by
$$ h(x) = \lim_{m \to \infty} (d(m,x) - m) . $$
\end{example}

\begin{figure}
\includegraphics[width=0.8\textwidth]{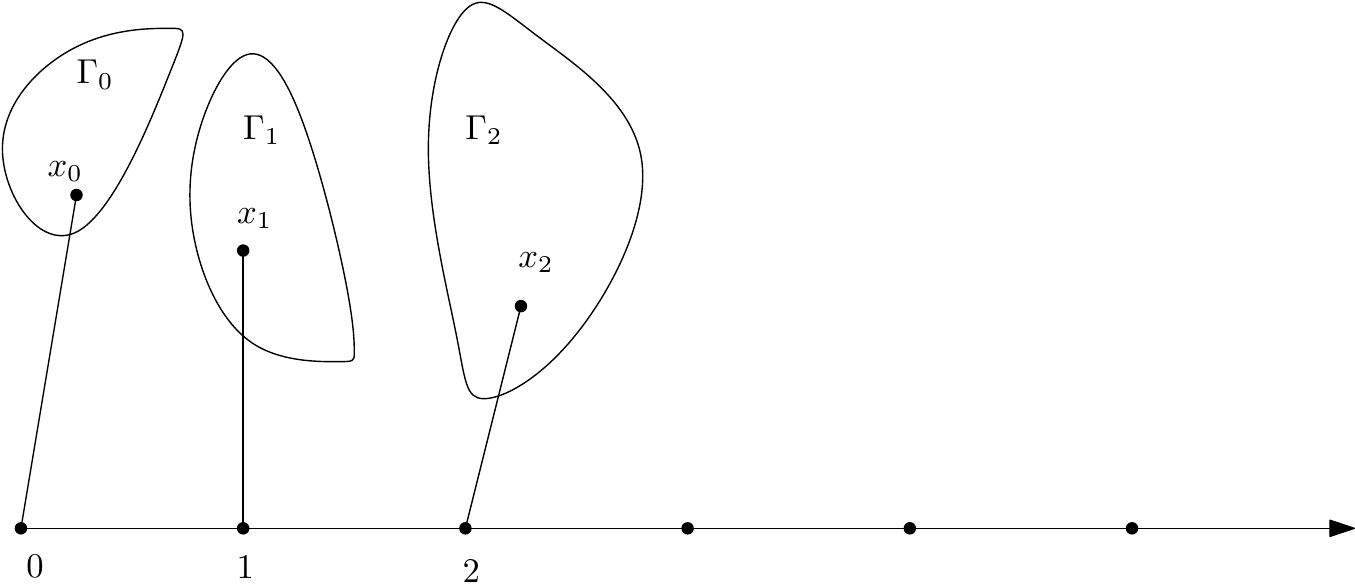}
\caption{The graph $\Gamma$ from Example \ref{exm:general graph exp growth}.}
\label{fig:grove}
\end{figure}


\begin{theorem}
\label{thm:horofunction are Busemann}
Let $G$ be a group that is virtually $\Z$,
and let $\Gamma(G,S)$ be any Cayley graph of $G$.

Then, any horofunction in $\p \Gamma(G,S)$
is a Busemann point; \ie $\p \Gamma(G,S) = \p_b \Gamma(G,S)$.
\end{theorem}

\begin{proof}
Since $G$ is virtually $\Z$ it has linear growth, so by Proposition \ref{prop:linear growth implies finite Busemann}
we know that the Busemann boundary $\p_b \Gamma(G,S)$ is finite.
$G$ acts on $\p_b \Gamma(G,S)$.  Let 
$$ K  = \{ x \in G \ : \ \forall \ h \in \p_b \Gamma(G,S) \ , \ x.h=h \} $$
be the kernel of the action.
Since $\p_b \Gamma(G,S)$ is finite, $[G:K] < \infty$.
By Lemma \ref{lemma:x for finite index} there exists $\vphi \in \p_b \Gamma(G,S)$ and $1 \neq y \in K$ such that 
$\vphi(y) = -|y|_S$.  
By Proposition \ref{prop:stabilizer}, $\vphi|_K$ is a homomorphism into $\R$, 
so $\vphi(x) = |x|_S$ for $x = y^{-1}$.
Since $\vphi$ is $1$-Lipschitz, by Lemma \ref{lemma:map and geodesics}
there exists a geodesic $\alpha$ in $\Gamma(G,S)$ such that
$\alpha_{n|x|} = x^n$ for all $n \in \N$.
Applying the same logic to $y=x^{-1}$ but with the $1$-Lipschitz function $-\vphi$,
we can also define a geodesic $\beta$ 
so that $\beta_{n|x|} = x^{-n}$ for all $n \in \N$.

Now, consider the subgroup $N = \IP{ x } \cong \Z$.  Since $G$ is virtually $\Z$, it is impossible that 
$[G:N] = \infty$. 
Let $R$ be a set of representatives for the cosets of $N$; 
\ie $G = \biguplus_{r \in R} r N$ with $|R| = [G:N]<\infty$. Assume that $1 \in R$.

Let $h \in \p \Gamma(G,S)$ be an arbitrary horofunction.
Let $(y_n)_n$ be a sequence in $G$ such that $h = \lim_{n \to \infty} b_{y_n}$.
By Proposition \ref{prop:unbounded horofunctions}, the set $\{ y_n \}_n$ is infinite, and thus, there must exist
some $r \in R$ such that infinitely many $y_n$ are in the coset $r N$.
Thus, one of the two sets
$$ \{ y_n \ : \ \exists \ k \ , \ y_n = r\alpha_k \} \qquad \textrm{ or } 
\{ y_n \ : \ \exists \ k \ , \ y_n = r\beta_k \} $$
is infinite.
That is, the sequence $(y_n)_n$ has infinitely many common points with one of the geodesics
$r \alpha$ or $r \beta$.
This implies that 
$$ h = \lim_{n\to \infty} b_{y_n} = \lim_{k \to \infty} b_{r x^k} \in \{ r.\alpha_\infty , r.\beta_\infty \} . $$
As this was for an arbitrary horofunction $h \in \p \Gamma(G,S)$, we have that
$$ \p \Gamma(G,S) \subset \{ r.\alpha_\infty , r.\beta_\infty \ : \ r \in R \} \subset \p_b \Gamma(G,S) . $$
\end{proof}

\end{document}